\documentclass[11pt,a4paper]{article}
\usepackage{epsf,epsfig,amsfonts,amsgen,amsmath,amstext,amsbsy,amsopn,amsthm,lineno}
\usepackage{color}
\usepackage{graphicx}
\newtheorem{thm}{Theorem}

\newtheorem{lemma}{Lemma}

\theoremstyle{definition}

\newtheorem{claim}{Claim}

\newtheorem{remark}[claim]{Remark}

\renewcommand{\mod}{{\rm mod}}

\begin{document}
\title
{\bf\Large Notes on $q$-analogues of a binomial congruence of Glaisher}

\date{}
\author{\small Bo Ning\thanks{E-mail address: ningbo\_math84@mail.nwpu.edu.cn (B. Ning)}\\[2mm]
\small Department of Applied Mathematics, School of Science,\\
\small Northwestern Polytechnical University,Xi'an, Shaanxi 710072, P.R.~China}
\maketitle
\begin{abstract}
Andrews once gave $q$-analogues of a binomial congruence of Glaisher, and he suggested perfect $q$-analogues. In this short note we give ones meeting the demand of Andrews.
\smallskip

\end{abstract}
An interesting result of Glaisher \cite{Glaisher} asserted that if $p\geq 5$ is a prime and $m\geq 1$ is an integer, then
\begin{align}\label{al1}
\binom{mp+p-1}{p-1}\equiv1 ~~(\mod~p^3).
\end{align}
Obviously, his result can be written as:
\begin{align}\label{al2}
(mp+1)(mp+2)\cdots (mp+p-1)\equiv(p-1)!~~(\mod~p^3).
\end{align}
As shown in \cite{Gasper_Rahman}, we define the \emph{$q$-number} by
\begin{align}\label{al3}
[n]_{q}=\frac{1-q^n}{1-q}~(q\neq1),
\end{align}
the \emph{$q$-factorial} by
\begin{align}\label{al4}
{[n]!}_q=[n]_q[n-1]_q\cdots[1]_q,
\end{align}
the \emph{$q$-binomial coefficient} by
\begin{align}\label{al5}
\binom{n}{k}_q=\frac{{[n]!}_q}{{[k]!}_q{[n-k]!}_q},
\end{align}
and the \emph{$q$-shifted factorial} by
\begin{align}\label{al6}
(A;q)_{n}=(1-A)(1-Aq)\cdots (1-Aq^{n-1}).
\end{align}
In view of $q$-analogues of Glaisher's congruence, Andrews \cite{Andrews} showed that if $p$ is an odd prime and $m\geq 1$ be an integer, then
\begin{align}\label{al7}
\binom{(m+1)p-1}{p-1}_q\equiv q^{mp(p-1)/2}~~~~(\mod~[p]_q^2).
\end{align}
The congruence (\ref{al7}) can be considered as a $q$-analogue of (\ref{al1}).
Furthermore, Andrews \cite{Andrews} suggested perfect analogues of (\ref{al1}) and (\ref{al2}), that is, $q$-congruences modulo $[p]_q^3$.

The main purpose of this short note is to extend Andrews' congruence (\ref{al7}) as follows.
\begin{thm}\label{th1}
Let $p\geq 5$ be a prime and $m\geq 1$ be an integer. Then
\begin{align}\label{al8}
&\binom{(m+1)p-1}{p-1}_q \nonumber\\
&\equiv q^{mp(p-1)/2}-(\frac{m^2(p^2-1)}{8}+\frac{m(p-1)(7p-5)}{24})(1-q)^2[p]_q^2~~(\mod~[p]_q^3).
\end{align}
\end{thm}

\begin{lemma}\label{le1}
Let $p\geq 5$ be a prime and $m\geq 1$ be an integer. Then
\begin{align}\label{al9}
\binom{(m+1)p}{p}_q\equiv \binom{m+1}{1}_{q^{p^2}}-\binom{m+1}{2}\frac{p^2-1}{12}(1-q)^2[p]_q^2 ~(\mod~[p]_q^3).
\end{align}
\end{lemma}

\begin{proof}
See Lemma 4 in \cite{Ning} or Theorem 1 in \cite{Straub}.
\end{proof}

\begin{lemma}\label{le1}
Let $p$ be an odd prime and $s\geq 1$ be an integer. Then
\begin{align}\label{al10}
q^{pi}\equiv\sum\limits_{k=0}^{s-1}\binom{i}{k}(-1)^{k}[p]_{q}^k(1-q)^k ~~(\mod~[p]_q^s).
\end{align}
\end{lemma}
\begin{proof}
\begin{align*}
q^{pi}&=(1-(1-q^p))^i=\sum\limits_{k=0}^{i}\binom{i}{k}(-1)^{k}(1-q^p)^{k}=\sum\limits_{k=0}^{i}\binom{i}{k}(-1)^{k}[p]_{q}^k(1-q)^k\\
&\equiv\sum\limits_{k=0}^{s-1}\binom{i}{k}(-1)^{k}[p]_{q}^k(1-q)^k ~~(\mod~[p]_q^s).
\end{align*}
\end{proof}

If we put $s=3$, then by Lemma \ref{le1}, we have
\begin{align}\label{al11}
q^{pi}&\equiv 1-i(1-q)[p]_q+\frac{i(i-1)}{2}(1-q)^2[p]_q^2~~(\mod~[p]_q^3).
\end{align}

\noindent{} {\bf {Proof of Theorem \ref{th1}.}}
\begin{align*}
&\binom{(m+1)p-1}{p-1}_q\\
&=\frac{[p]_q}{[(m+1)p]_q}\binom{(m+1)p}{p}_q\\
&\equiv\frac{1-q^p}{1-q^{(m+1)p}}(\frac{1-q^{(m+1)p^2}}{1-q^{p^2}}-\binom{m+1}{2}\frac{p^2-1}{12}(1-q)^2[p]_q^2)~~(\mod~[p]_q^3)~~~~~~~~(by~(\ref{al9}))\\
&\equiv\frac{(1-q)[p]_q}{(m+1)(1-q)[p]_q-\frac{(m+1)m}{2}(1-q)^2[p]_q^2}\cdot(\frac{(m+1)p(1-q)[p]_q-\frac{(m+1)p(mp+p-1)}{2}(1-q)^2[p]_q^2}{p(1-q)[p]_q-\frac{p(p-1)}{2}(1-q)^2[p]_q^2}\\
&-\frac{m(m+1)}{2}\frac{p^2-1}{12}(1-q)^2[p]_q^2)~~(\mod~[p]_q^3)~~~~~~~~~(by~(\ref{al11}))\\
&\equiv\frac{1}{1-\frac{m}{2}(1-q)[p]_q}\cdot(\frac{1-\frac{mp+p-1}{2}(1-q)[p]_q}{1-\frac{p-1}{2}(1-q)[p]_q}-\frac{m(p^2-1)}{24}(1-q)^2[p]_q^2)~~(\mod~[p]_q^3)\\
&\equiv(1+\frac{m}{2}(1-q)[p]_q+\frac{m^2}{4}(1-q)^2[p]_q^2)\cdot(1-\frac{mp}{2}(1-q)[p]_q-\frac{m(7p^2-6p-1)}{24}(1-q)^2[p]_q^2)\\
&~~~~~~(\mod~[p]_q^3)\\
&\equiv1-\frac{m(p-1)}{2}(1-q)[p]_q+(\frac{m^2}{4}-\frac{m^2p}{4}-\frac{m(7p^2-6p-1)}{24})(1-q)^2[p]_q^2~~(\mod~[p]_q^3)\\
&\equiv q^{mp(p-1)/2}-(\frac{m^2(p^2-1)}{8}+\frac{m(p-1)(7p-5)}{24})(1-q)^2[p]_q^2~~(\mod~[p]_q^3).
\end{align*}
The proof is complete. {\hfill$\Box$}

\begin{remark}
In \cite{Andrews}, Andrews also gave a $q$-analogue of (\ref{al2}). He proved that if $p$ is an odd prime and $m\geq 1$, then
\begin{align}\label{al12}
\frac{{(q^{mp+1};q)_{p-1}}-q^{mp(p-1)/2}(q;q)_{p-1}}{(1-q^{(m+1)p})(1-q^{mp})}\equiv\frac{(p^2-1)p}{24}~~(\mod~[p]_q).
\end{align}
In \cite{Andrews}, Andrews once wrote the following words: "the empirical studies I made which led to Theorem 1 do not yield any instances of congruences modulo $[p]_q^3$". However, his studies surely can imply a $q$-analogue of congruence (\ref{al2}) modulo $[p]_q^3$. Noting that by Lemma \ref{le1}, we have $1-q^{(m+1)p}\equiv (m+1)(1-q)[p]_q ~(\mod~[p]_q^2)$ and $1-q^{mp}\equiv m(1-q)[p]_q ~(\mod~[p]_q^2)$. And Andrews' congruence (\ref{al12}) implies
\begin{thm}
\begin{align}
(q^{mp+1};q)_{p-1}-q^{mp(p-1)/2}(q;q)_{p-1}\equiv\frac{m(m+1)(p^2-1)p}{24}(1-q)^2[p]_q^2~~(\mod~[p]_q^3).
\end{align}
\end{thm}
\end{remark}

\section*{Acknowledgement}
This work is supported by the Doctorate Foundation of Northwestern Polytechnical University (cx201326).


\begin{thebibliography}{99}
\bibitem{Andrews}
G.E. Andrews, $q$-analogs of the binomial coefficient congruences of Babbage, Wolstenholme and Glaisher, \emph{Discrete Math.} {\bf 204} (1999) 15-25.

\bibitem{Gasper_Rahman}
G. Gasper and M. Rahman, Basic Hypergeometric Series, Cambridge University Press, Cambridge, (1990).

\bibitem{Glaisher}
J.W.L. Glaisher, Residues of binomial theorem coefficients with respect to $p^3$, \emph{Quart. J. Math., Oxford Series} {\bf 31} (1900) 110-124.

\bibitem{Ning}
B. Ning, An inductive proof of Straub's $q$-analogue of Ljunggren's congruence, arXiv:1301.2986v2.

\bibitem{Straub}
A. Straub, A $q$-analogue of Ljunggren's binomial congrunence, \emph{Discrete Mathematics and Theoretical Computer Science (DMTCS)}, Nancy, France, proc. AO (2011) 897-
902, arXiv:1103.3258v1 [math. NT].
\end{thebibliography}
\end{document}